\documentclass[12pt]{article}

\newtheorem{theorem}{Theorem}

\newtheorem{lemma}{Lemma}

\usepackage{amscd}

\usepackage{epsfig}
\usepackage{url}
\usepackage{amssymb}
\usepackage{amsmath}
\usepackage{amsfonts}

\def\Dbar{\leavevmode\lower.6ex\hbox to 0pt{\hskip-.23ex \accent"16\hss}D}

\def\bR{{\mbox{\bf R}}}
\def\bZ{{\mbox{\bf Z}}}
\def\bC{{\mbox{\bf C}}}

\def\paf{{\mbox{\rm PAF}}}
\def\psd{{\mbox{\rm PSD}}}
\def\dft{{\mbox{\rm DFT}}}

\begin{document}

{\bf\LARGE
\begin{center}
There is no circulant weighing matrix of order 60 and
weight 36
\end{center}
}

\begin{abstract}
With the help of a computer, we prove the assertion made in the title. 
\end{abstract}

{\Large
\begin{center}
Dragomir {\v{Z}}. {\Dbar}okovi{\'c}\footnote{University of Waterloo,
Department of Pure Mathematics, Waterloo, Ontario, N2L 3G1, Canada
e-mail: \url{djokovic@math.uwaterloo.ca}}, Ilias S.
Kotsireas\footnote{Wilfrid Laurier University, Department of Physics
\& Computer Science, Waterloo, Ontario, N2L 3C5, Canada, e-mail:
\url{ikotsire@wlu.ca}}
\end{center}
}

\section{Introduction}

A {\em weighing matrix} of {\em order} $n$ and {\em weight} $w>0$ 
is a matrix $W$ of order $n$ whose entries belong to $\{0,\pm1\}$ 
such that $WW^T=wI$. (T denotes the transpose and $I$ the
identity matrix.) For more information on weighing matrices see 
\cite{CK:2007}. We use the symbol $CW(n,w)$ to refer to a weighing matrix of order $n$ and weight $w$ which is also a circulant. It is well known that if $W$ is a $CW(n,w)$ then 
$w=a^2$ where $a$ is the sum of the first row of $W$.
The existence question for $CW(60,36)$ was raised in Strassler's table \cite{Strassler} and has remained open for a long time \cite{AG:2010}. We shall prove that $CW(60,36)$ does not exist.

\section{Preliminaries}

Let $\bZ_n=\bZ/n\bZ=\{0,1,\ldots,n-1\}$ be the ring of integers modulo $n$. Let $A=[a_0,a_1,\ldots,a_{n-1}]$ be an integer sequence of length $n$. We view the indexes $0,1,\ldots,n-1$ as elements of $\bZ_n$. 
The {\em periodic autocorrelation function} of $A$ is the real-valued function $\paf_A:\bZ_n\to\bR$ defined by 
\begin{equation} \label{paf-def}
\paf_A(s)=\sum_{i=0}^{n-1} a_i a_{i+s}.
\end{equation}
(The indexes should be reduced modulo $n$.)

Let us now recall some basic facts about a $CW(n,w)$ matrix $W$. 
Let $A=[a_0,a_1,\ldots,a_{n-1}]$ be its first row and 
$a=\sum_i a_i$. As mentioned above we have $w=a^2$. By replacing $W$ with $-W$ if necessary, we may assume that $a>0$. 
It follows from $WW^T=wI$ that 
\begin{equation} \label{paf-eq}
\paf_A(s)=0, \quad s=1,2,\ldots,n-1.
\end{equation}
By abuse of language, we say that the periodic autocorrelation of $A$ is zero although $\paf_A(0)=w$ is not zero. The converse is also valid: if $A$ is a ternary sequence (i.e., a sequence with entries $0,\pm1$) of length $n$ and weight $w$ having zero periodic autocorrelation, then the corresponding circulant matrix $W$ is a $CW(n,w)$.

The {\em discrete Fourier transform} of $A$ is the complex-valued function $\dft_A:\bZ_n\to\bC$ defined by
\begin{equation} \label{dft-def}
\dft_A(s)=\sum_{j=0}^{n-1} a_j \omega^{js},
\end{equation}
where $\omega=e^{2\pi i/n}$. The {\em power spectral density} of $A$ is the real-valued function $\psd_A:\bZ_n\to\bR$ defined by 
\begin{equation} \label{psd-def}
\psd_A(s)=\left| \dft_A(s) \right|^2.
\end{equation}

We also need the notion of equivalence of integer sequences 
of fixed length, say $n$. Denote by $X_n$ the set of such 
sequences. Let $\bZ_n^*$ be the group of units (i.e., invertible elements) of $\bZ_n$. Let $\Phi_n$ be the group of affine transformations $\sigma:\bZ_n\to\bZ_n$ sending $i\mapsto ui+v$ where $u\in\bZ_n^*$, $v\in\bZ_n$. The correspondence 
$\sigma\leftrightarrow\left(\begin{array}{cc}u&v\\0&1\end{array}
\right)$ is an isomorphism of $\Phi_n$ with the group of these 
$2\times2$ matrices over $\bZ_n$. Thus, $\Phi_n$ is a semidirect product of $\bZ_n$ and $\bZ_n^*$. The action of 
$\Phi_n$ on $X_n$ is defined as follows. If $\sigma\in\Phi_n$ and $x=[x_0,x_1,\ldots,x_{n-1}]\in X_n$ then $\sigma(x)=
[x_{\sigma^{-1}(0)},x_{\sigma^{-1}(1)},\ldots,x_{\sigma^{-1}(n-1)}]$. We say that two sequences in $X_n$ are {\em equivalent} if they belong to the same orbit of $\Phi_n$.

The basic fact that we need (which is easy to verify) is the following: if a sequence $x\in X_n$ has a zero periodic autocorrelation, then the same is true for any sequence in the equivalence class of $x$.

Given a factorization $n=md$, we define the $m$-{\em compression map} $c_{n,d}:X_n\to X_d$ as the map sending a sequence 
$x=[x_0,x_1,\ldots,x_{n-1}]$ to the sequence 
$y=[y_0,y_1,\ldots,y_{d-1}]$, where 
$y_i=x_i+x_{i+d}+\cdots+x_{i+(m-1)d}$. It is known that if $x$ has zero periodic autocorrelation, then so does $y$ (see \cite{DK:JCD:2014}). The reduction map $\bZ_n^*\to\bZ_d^*$ modulo $d$ is a surjective group homomorphism, i.e., for each $u\in\bZ_d^*$ there exists $\tilde{u}\in\bZ_n^*$ such that $\tilde{u}=u \pmod{d}$. 

\begin{lemma} \label{dijagram}
Let $\sigma\in\Phi_d$ be given by $\sigma(j)=uj+v$ with $u\in\bZ_d^*$, $v\in\bZ_d$.
Let $\tilde{\sigma}\in\Phi_n$ be given by $\tilde{\sigma}(i)=
\tilde{u}i+\tilde{v}$, where $\tilde{u}=u \pmod{d}$ and $\tilde{v}=v \pmod{d}$.
Then the following diagram commutes

\begin{eqnarray} \label{diagram}
\begin{CD}
X_n @> c_{n,d} >> X_d \\
@V \tilde{\sigma}^{-1} VV @VV \sigma^{-1} V\\
X_n @> c_{n,d} >> X_d
\end{CD}
\end{eqnarray}
\end{lemma}

\begin{proof}
Let $x=[x_i]_{i=0}^{n-1}\in X_n$ be arbitrary and let
$y=c_{n,d}(x)$, i.e., $y=[y_j]_{j=0}^{d-1}\in X_d$ where  $y_j=x_j+x_{j+d}+\cdots+x_{j+(m-1)d}$ for $j\in\bZ_d$. 
Let $x'=\tilde{\sigma}^{-1}(x)$ and $y'=\sigma^{-1}(y)$. Thus 
$x'_i=x_{\tilde{u}i+\tilde{v}}$ for $i\in\bZ_n$ and $y'_j=y_{uj+v}$ for $j\in\bZ_d$. We have to show that $y'=c_{n,d}(x')$, i.e.,
\begin{equation}
\sum_{k=0}^{m-1} x_{uj+v+kd}=
\sum_{k=0}^{m-1} x_{\tilde{u}(j+kd)+\tilde{v}},\quad j\in\bZ_d.
\end{equation}
Equivalently, we have to show that the two subsets of $\bZ_n$, 
namely, $\{ uj+v+kd:k\in\bZ_m \}$ and $\{ \tilde{u}j+\tilde{v}+
k\tilde{u}d:k\in\bZ_m \}$ are equal. As $\tilde{u}$ is relatively prime to $n$, both subsets are cosets of the additive subgroup of $\bZ_n$ generated by $d$. Since $\tilde{u}j+\tilde{v}=uj+v \pmod{d}$, the two  cosets are equal.
\end{proof}

\section{Main result}

We can now prove our main result. Our proof is computational and based on the compression method (see \cite{DK:JCD:2014}). All computations were performed on the saw.sharcnet.ca cluster running at 2.83 GHz.

\begin{theorem}
There is no circulant weighing matrix of order $60$ and 
weight $36$.
\end{theorem}

\begin{proof}
For convenience we set $n=60$ and $w=36$. Assume that there exists a $CW(n,w)$, say $W$. Let $A=[a_0,a_1,\ldots,a_{n-1}]$ be the first row of $W$ and we set $a=\sum_i a_i$.  Let $p,q,r$ denote the number of terms $a_i$ which are equal to $0,1,-1$, respectively. Since $a^2=w$ and we can assume that $a>0$, we have $q-r=a=6$. As $q+r=w$, we have $p=24$, $q=21$, and $r=15$.

We choose the factorization $n=md$ with $m=3$ and $d=20$. 
By compressing $A$ using the compression factor $m=3$, we obtain 
the sequence $B=[b_0,b_1,\ldots,b_{d-1}]$ where 
$b_i=a_i+a_{i+d}+a_{i+2d}$, $0\le i<d$. By \cite[Theorem 3]{DK:JCD:2014}, $B$ has also periodic autocorrelation zero. 
Note that each $b_i$ belongs to the set $Z=\{0,\pm1,\pm2,\pm3\}$. For any sequence  $[x_0,\ldots,x_{d-1}]$ with $x_i\in Z$ we define its {\em content} to be the sequence $\mu=[\mu_0,\mu_1,\mu_{-1},\mu_2,\mu_{-2},\mu_3,\mu_{-3}]$, where $\mu_j$ is the number of indexes $i$ such that $x_i=j$.  For convenience, we also set $\nu_0=\mu_0$ and $\nu_i=\mu_i+\mu_{-i}$ for $i=1,2,3$. 
For a fixed content $\mu$, we denote by $X_\mu$ the subset of $X_d$ consisting of all sequences having the content $\mu$. Note that equivalent sequences have the same content, and so $X_\mu$ is a union of equivalence classes.

Let $\mu^B$ denote the content of $B$. Let us also set $\nu_0^B=\mu_0^B$ and $\nu_i^B=\mu_i^B+\mu_{-i}^B$ for $i=1,2,3$. Obviously we have $\sum_i \nu_i^B=d$, i.e., the $\nu_i^B$ satisfy the equation
\begin{equation} \label{prva-eq}
\nu_0+\nu_1+\nu_2+\nu_3=20.
\end{equation}

Since $A$ has zero periodic autocorrelation, by \cite[Theorem 2]{DK:JCD:2014} we have
\begin{equation} \label{psd-eq}
\psd_A(s)=w=36, \quad s=0,1,2,\ldots,n-1.
\end{equation}
By \cite[Theorem 3]{DK:JCD:2014} we also have
\begin{equation} \label{psd-eq-c}
\psd_B(s)=w=36, \quad s=0,1,2,\ldots,d-1.
\end{equation}
By setting $s=1$ and $\omega_0=e^{2\pi i/d}$ and by using the fact that $B$ has zero periodic autocorrelation, we obtain that
\begin{eqnarray*}
\psd_B(1) &=& \left| \sum_{j=0}^{d-1} b_j \omega_0^j \right|^2 
= \sum_{j=0}^{d-1} \sum_{k=0}^{d-1} b_jb_k\omega_0^{j-k} \\
&=& \sum_{l=0}^{d-1} \left( \sum_{k=0}^{d-1} b_k b_{k+l} \right) 
\omega_0^l 
= \sum_{k=0}^{d-1} b_k^2.
\end{eqnarray*}
Thus $\sum_i i^2\nu_i^B=w$, i.e., the $\nu_i^B$ satisfy the equation
\begin{equation} \label{druga-eq}
\nu_1+4\nu_2+9\nu_3=36.
\end{equation}

Consequently, $\mu^B$ is a nonnegative integral solution of the following system of three linear Diophantine equations:
\begin{eqnarray} 
\mu_0+\mu_1+\mu_{-1}+\mu_2+\mu_{-2}+\mu_3+\mu_{-3} &=&                                    20, \label{prva} \\
(\mu_1-\mu_{-1})+2(\mu_2-\mu_{-2})+3(\mu_3-\mu_{-3}) &=& 6, 
\label{druga} \\
(\mu_1+\mu_{-1})+4(\mu_2+\mu_{-2})+9(\mu_3+\mu_{-3}) &=& 36.
\label{treca}
\end{eqnarray}

By a straightforward computer enumeration, we found that there are exactly 76 different contents $\mu$ satisfying this system of equations. They are listed in the appendix. Hence, $\mu^B$ must be one of these 76 contents. In order to show that the zero autocorrelation property is violated for all sequences in some $X_\mu$, it suffices to do that for the representatives of the equivalence classes contained in $X_\mu$. As a representative of an equivalence class we choose the sequence in the class which is the smallest in the lexicographic ordering. We consider the integers in $Z$ as denoting 7 different colors, and we order them so that $0<1<-1<2<-2<3<-3$. (The above-mentioned lexicographic ordering uses this ordering of the colors.) We refer to these representatives as {\em charm bracelets} with content $\mu$. An algorithm for generating the charm bracelets of fixed content is presented in the recent preprint \cite{DKRS}. We have applied this algorithm to each of the 76 possible contents $\mu$. The total cpu time used for these computations was about 13 hours and 10 minutes. 

The upshot is that altogether there are only four charm bracelets with zero autocorrelation function, namely the following:
\begin{eqnarray*} 
B_1 &=& [1,1,1,1,-1,1,-1,-1,3,1,-1,-1,-1,-1,1,-1,1,1,3,-1],\\
B_2 &=& [0,0,0,0,0,0,0,0,3,3,0,0,0,0,0,0,0,0,3,-3],\\
B_3 &=& [0,0,0,0,0,0,0,3,0,3,0,0,0,0,0,0,0,3,0,-3],\\
B_4 &=& [0,0,0,0,0,0,3,0,0,3,0,0,0,0,0,0,3,0,0,-3].
\end{eqnarray*}
The charm bracelet $B_1$ has content $[0,9,9,0,0,2,0]$, while 
$B_2,B_3$ and $B_4$ have the same content, namely $[16,0,0,0,0,3,1]$. Thus, $B$ must be equivalent to one of the $B_k$s, i.e., there exist $\sigma\in\Phi_d$ and $k\in\{1,2,3,4\}$ such that $\sigma^{-1}(B)=B_k$. Say, $\sigma(j)=uj+v$ for all $j\in\bZ_d$, where $u\in\bZ_d^*$ and $v\in\bZ_d$ are fixed.
Choose $\tilde{u}\in\bZ_n^*$ and $\tilde{v}\in\bZ_n$ such that $\tilde{u}=u \pmod{d}$ and $\tilde{v}=v \pmod{d}$. Define $\tilde{\sigma}\in\Phi_n$ by $\tilde{\sigma}(i)=\tilde{u}i+\tilde{v}$. If $A_k=(\tilde\sigma)^{-1}(A)$ then $c_{n,d}(A_k)=B_k$ by Lemma \ref{dijagram}. Thus, by replacing $A$ with $A_k$ we may assume that $B=B_k$.

There are exactly $6^{18}=101559956668416$ ternary sequences in $X_n$ whose $m$-compression is $B_1$. Also, there are exactly  
$7^{16}=33232930569601$ ternary sequences in $X_n$ whose $m$-compression is $B_k$, $k=2,3,4$. We have checked with a computer that none of the mentioned ternary sequences has zero periodic autocorrelation. The computation was carried out separately for each of the $B_k$s. 
For $k=1$ we divided the task to 36 processors, so that each processor had to check $6^{16}=2821109907456$ sequences. The total cpu time used by them was $1144.8$ hours. For $k>1$ we divided the task to 49 processors, so that each processor had to check $7^{14}=678223072849$ sequences. The total time used was $383.6$, $378.5$, $382.0$ hours for $k=2,3,4$, respectively.

Since there are no ternary sequences $x\in X_n$ with zero periodic autocorrelation such that $c_{n,d}(x)$ is one of the 
sequences $B_1,B_2,B_3,B_4$, we have a contradiction. 
Hence, we conclude that there are no $CW(60,36)$.
\end{proof}

\section{Acknowledgements}
The authors wish to acknowledge generous support by NSERC.
This work was made possible by the facilities of the Shared Hierarchical Academic Research Computing Network (SHARCNET) and Compute/Calcul Canada. We also thank Dan Recoskie and Joe Sawada 
for the use of their code for generating the charm bracelets.

\section*{Appendix: The 76 contents $\mu$}

We list all 76 nonnegative integral solutions of the system 
of three linear Diophantine equations (\ref{prva})-(\ref{treca}).

$$
\begin{array}{lll}
0,12,6,0,0,1,1 & 0,15,3,0,0,0,2 & 0,9,9,0,0,2,0 \\ 
10,0,3,6,0,0,1 & 10,1,2,4,2,1,0 & 10,2,1,5,1,0,1 \\
10,3,0,3,3,1,0 & 11,0,0,6,3,0,0 & 11,0,5,1,0,3,0 \\
1,10,5,1,2,1,0 & 1,11,4,2,1,0,1 & 11,2,3,0,1,3,0 \\
1,12,3,0,3,1,0 & 11,3,2,1,0,2,1 & 1,13,2,1,2,0,1 \\ 
11,5,0,0,1,2,1 & 1,15,0,0,3,0,1 & 12,0,2,4,0,1,1 \\
12,1,1,2,2,2,0 & 12,2,0,3,1,1,1 & 14,0,1,2,0,2,1 \\
14,1,0,0,2,3,0 & 16,0,0,0,0,3,1 & 1,6,9,3,0,1,0 \\
1,8,7,2,1,1,0 & 1,9,6,3,0,0,1 & 2,11,1,2,4,0,0 \\
2,3,9,6,0,0,0 & 2,5,7,5,1,0,0 & 2,7,5,4,2,0,0 \\
2,9,3,3,3,0,0 & 3,11,3,0,1,1,1 & 3,12,2,1,0,0,2 \\
3,14,0,0,1,0,2 & 3,6,8,1,0,2,0 & 3,8,6,0,1,2,0 \\
3,9,5,1,0,1,1 & 4,10,1,2,2,0,1 & 4,11,0,0,4,1,0 \\
4,3,8,4,0,1,0 & 4,5,6,3,1,1,0 & 4,6,5,4,0,0,1 \\
4,7,4,2,2,1,0 & 4,8,3,3,1,0,1 & 4,9,2,1,3,1,0 \\
5,0,8,7,0,0,0 & 5,2,6,6,1,0,0 & 5,4,4,5,2,0,0 \\
5,6,2,4,3,0,0 & 5,8,0,3,4,0,0 & 6,10,0,0,2,1,1 \\
6,3,7,2,0,2,0 & 6,5,5,1,1,2,0 & 6,6,4,2,0,1,1 \\
6,7,3,0,2,2,0 & 6,8,2,1,1,1,1 & 6,9,1,2,0,0,2 \\
7,0,7,5,0,1,0 & 7,2,5,4,1,1,0 & 7,3,4,5,0,0,1 \\
7,4,3,3,2,1,0 & 7,5,2,4,1,0,1 & 7,6,1,2,3,1,0 \\
7,7,0,3,2,0,1 & 8,1,3,6,2,0,0 & 8,3,1,5,3,0,0 \\
8,3,6,0,0,3,0 & 8,6,3,0,0,2,1 & 8,9,0,0,0,1,2 \\
9,0,6,3,0,2,0 & 9,2,4,2,1,2,0 & 9,3,3,3,0,1,1 \\
9,4,2,1,2,2,0 & 9,5,1,2,1,1,1 & 9,6,0,0,3,2,0 \\
9,6,0,3,0,0,2 &               &
\end{array}
$$

\end{document}